\documentclass[12pt]{amsart}
\pdfoutput=1

\usepackage[paperheight=11in,paperwidth=8.5in,left=1in,top=1.25in,right=1in,bottom=1.25in]{geometry}
\usepackage[pdfstartpage={1},pdfstartview={FitH},pdftitle={},pdfauthor={},pdfsubject={},pdfcreator={},pdfproducer={},pdfkeywords={},pagebackref=false,colorlinks=true,linkcolor=black,citecolor=black,filecolor=black,urlcolor=black,]{hyperref}
\usepackage{afterpage}
\usepackage{amsfonts}
\usepackage{amsmath}
\allowdisplaybreaks[4]
\usepackage{amssymb}
\usepackage{amsthm}
\usepackage{array}
\usepackage{bbm}
\usepackage[justification=centering]{caption}
\usepackage[capitalize,nameinlink,noabbrev,compress]{cleveref}
\usepackage{enumerate}
\usepackage{float}
\restylefloat{figure}
\usepackage{graphicx}
\usepackage{mathrsfs} 
\usepackage{mathtools}
\usepackage[framemethod=tikz,ntheorem,xcolor]{mdframed}
\usepackage{parcolumns}
\usepackage{tabularx}
\usepackage{tikz}\pdfpageattr{/Group <</S /Transparency /I true /CS /DeviceRGB>>}
\usetikzlibrary{arrows,calc,intersections,matrix,positioning,through}
\usepackage{tikz-cd}
\usepackage[all]{xy}
\usepackage[aligntableaux=center]{ytableau}

\numberwithin{equation}{section}

\newtheorem{thm}[equation]{Theorem}
\AfterEndEnvironment{thm}{\noindent\ignorespaces}

\AfterEndEnvironment{conj}{\noindent\ignorespaces}
\newtheorem{cor}[equation]{Corollary}
\AfterEndEnvironment{cor}{\noindent\ignorespaces}
\newtheorem{lem}[equation]{Lemma}
\AfterEndEnvironment{lem}{\noindent\ignorespaces}

\AfterEndEnvironment{prob}{\noindent\ignorespaces}
\newtheorem{prop}[equation]{Proposition}
\AfterEndEnvironment{prop}{\noindent\ignorespaces}

\theoremstyle{definition}
\newtheorem{defn}[equation]{Definition}
\AfterEndEnvironment{defn}{\noindent\ignorespaces}
\newtheorem*{pf_no_qed}{Proof}

\AfterEndEnvironment{proof}{\noindent\ignorespaces}
\newtheorem{eg_no_qed}[equation]{Example}

\newenvironment{eg}[1][]{\begin{eg_no_qed}[#1]\renewcommand{\qedsymbol}{$\Diamond$}\pushQED{\qed}}{\popQED\end{eg_no_qed}}
\AfterEndEnvironment{eg}{\noindent\ignorespaces}

\AfterEndEnvironment{rmk}{\noindent\ignorespaces}

\theoremstyle{remark}
\newtheorem*{claim}{Claim}
\AfterEndEnvironment{claim}{\noindent\ignorespaces}
\newtheorem*{claimpf_no_qed}{Proof of Claim}
\newenvironment{claimpf}[1][]{\begin{claimpf_no_qed}[#1]\renewcommand{\qedsymbol}{$\square$}\pushQED{\qed}}{\popQED\end{claimpf_no_qed}}
\AfterEndEnvironment{claimpf}{\noindent\ignorespaces}
 
\DeclareMathOperator{\End}{End}

\DeclareMathOperator{\GL}{GL}
\DeclareMathOperator{\gl}{\mathfrak{gl}}
\DeclareMathOperator{\Gr}{Gr}

\DeclareMathOperator{\interior}{int}

\DeclareMathOperator{\Mat}{Mat}

\newcommand{\pref}[1]{\hyperref[#1]{\pageref{#1}}}

\newcommand{\rf}[1]{\hyperref[#1]{(\ref*{#1})}}

\DeclareMathOperator{\spn}{span}

\def\R{{\mathbb R}}
\def\RR{{\mathbb R}}
\def\CC{{\mathbb C}}
\def\RP{{\mathbb P}}

\def\X{{\mathcal{X}}}

\def\Grtnn{\Gr_{\ge 0}(k,n)}
\def\Grtp{\Gr_{>0}(k,n)}

\newcommand{\exps}[1]{\exp(#1\shift)}
\def\expts{\exps{t}}

\def\shift{{S}}

\def\cs{cyclically symmetric }
\def\ampl{\mathcal{A}_{n,k,m}(Z)}
\def\Zsym{Z_0}
\def\amplZsym{\mathcal{A}_{n,k,m}(\Zsym)}

\def\Id{\operatorname{Id}}
\def\Z{{\mathbb Z}}

\def\Plucker{\Delta}

\def\Lin{{\RR^{\binom{2n}{n-1}}}}
\def\Proj{{\RP^{\binom{2n}{n-1}-1}}}
\def\e{e}
\def\odd{[2n]_{\operatorname{odd}}}
\def\even{[2n]_{\operatorname{even}}}
\def\NC{\mathcal{NC}_n}
\def\sigmadual{\tilde\sigma}

\def\Hspace{\mathcal{H}}
\def\A{A}
\def\Grrr{\Gr(n-1,2n)}
\def\X{\mathcal{X}}
\def\Grtn{\Gr_{\ge 0}(n-1,2n)}

\def\Up{u_i}
\def\Down{d_i}
\def\isolate{\sigma'(i)}
\newcommand\isol[1]{\sigma'(#1)}
\def\UD{\Phi}

\def\openball{Q}
\def\closedball{\overline{\openball}}
\def\affinespan{R}
\def\flow{f}
\newcommand\act[2]{\flow(#1,#2)}

\def\eigtau{u}
\def\l{\lambda}
\def\mat{\Mat(k,n-k)}
\def\expttau{\exp(t\tau)}

\def\proj{\pi}
\def\emb{\gamma}
\def\ZsymGr{{(Z_0)_{\Gr}}}
\def\openball{Q_0}
\def\closedball{\overline{Q_0}}
\def\openball{Q}
\def\closedball{\overline{Q}}
\def\bigpart{\kappa}
\def\smallpart{\mu}

\title{The totally nonnegative Grassmannian is a ball}

\author{Pavel Galashin}
\address{Department of Mathematics, University of California, Los Angeles, 520 Portola Plaza, Math Sciences Building, Los Angeles, CA 90095, USA}
\email{\href{mailto:galashin@math.ucla.edu}{galashin@math.ucla.edu}}
\author{Steven N. Karp}
\address{LaCIM, Universit\'{e} du Qu\'{e}bec \`{a} Montr\'{e}al, CP 8888, Succ.\ Centre-ville, Montr\'{e}al, QC H3C 3P8, Canada}
\email{\href{mailto:karp.steven@courrier.uqam.ca}{karp.steven@courrier.uqam.ca}}
\author{Thomas Lam}
\address{Department of Mathematics, University of Michigan, 2074 East Hall, 530 Church Street, Ann Arbor, MI 48109-1043, USA}
\email{\href{mailto:tfylam@umich.edu}{tfylam@umich.edu}}
\thanks{P.G.\ was supported by an Alfred P. Sloan Research Fellowship and by the National Science Foundation under Grants No.~DMS-1954121 and No.~DMS-2046915. S.N.K.\ was supported by the National Science Foundation under Grant No.~DMS-1600447 and by the Natural Sciences and Engineering Research Council of Canada under a Postdoctoral Fellowship. T.L.\ was supported by a von Neumann Fellowship from the Institute for Advanced Study and by the National Science Foundation under Grants No.~DMS-1464693 and No.~DMS-1953852.}

\begin{document}

\begin{abstract}
We prove that three spaces of importance in topological combinatorics are homeomorphic to closed balls: the totally nonnegative Grassmannian, the compactification of the space of electrical networks, and the cyclically symmetric amplituhedron.
\end{abstract}

\date{\today}

\maketitle

\section{Introduction}\label{sec:intro}

\noindent The prototypical example of a closed ball of interest in topological combinatorics is a convex polytope. Over the past few decades, an analogy between convex polytopes, and certain spaces appearing in total positivity and in electrical resistor networks, has emerged~\cite{Lus2,FS,Pos,CGV,CIM}. One motivation for this analogy is that these latter spaces come equipped with cell decompositions whose face posets share a number of common features with the face posets of polytopes~\cite{Wil,Hersh,RW}. A new motivation for this analogy comes from recent developments in high-energy physics, where physical significance is ascribed to certain differential forms on positive spaces which generalize convex polytopes~\cite{abcgpt,AT,ABL}. In this paper we show in several fundamental cases that this analogy holds at the topological level: the spaces themselves are closed balls.

\subsection{The totally nonnegative Grassmannian}\label{intro:Gr}

Let $\Gr(k,n)$ denote the Grassmannian of $k$-planes in $\RR^n$.  Postnikov~\cite{Pos} defined its totally nonnegative part $\Grtnn$ as the set of $X\in\Gr(k,n)$ whose Pl\"ucker coordinates are all nonnegative.  The totally nonnegative Grassmannian is not a polytope, but Postnikov conjectured that it is the `next best thing', namely, a regular CW complex homeomorphic to a closed ball.  He found a cell decomposition of $\Gr_{\ge 0}(k,n)$, where each open cell is specified by requiring some subset of the Pl\"{u}cker coordinates to be strictly positive, and requiring the rest to equal zero.

Over the past decade, much work has been done towards Postnikov's conjecture. The face poset of the cell decomposition (described in~\cite{Rietsch,rietsch06,Pos}) was shown to be shellable by Williams~\cite{Wil}. Postnikov, Speyer, and Williams~\cite{PSW} showed that the cell decomposition is a CW complex, and Rietsch and Williams~\cite{RW} showed that it is regular up to homotopy, i.e., the closure of each cell is contractible. Our first main theorem is: 
\begin{thm}\label{thm:Gr}
The space $\Grtnn$ is homeomorphic to a $k(n-k)$-dimensional closed ball.
\end{thm}

It remains an open problem to establish Postnikov's conjecture, i.e., to address arbitrary cell closures in the cell decomposition of $\Grtnn$.\footnote{Since this paper was completed, we verified Postnikov's conjecture using different methods \cite{GKL}.} Each of Postnikov's cells determines a matroid known as a \emph{positroid}, and \cref{thm:Gr} also reflects how positroids are related via specialization (see~\cite{ARW} for a related discussion about oriented matroids).

Separately, Lusztig~\cite{Lus2} defined and studied the totally nonnegative part $(G/P)_{\geq 0}$ of a partial flag variety of a split real reductive group $G$.  In the case $G/P = \Gr(k,n)$, Rietsch showed that Lusztig's and Postnikov's definitions of the totally nonnegative part are equivalent (see e.g.~\cite[Remark 3.8]{LamCDM} for a proof). Lusztig~\cite{LusIntro} showed that $(G/P)_{\geq 0}$ is contractible, and our approach to \cref{thm:Gr} is similar to his (see \cref{sec:remarks}).  
We discuss the case $(G/P)_{\geq 0}$ in a separate work~\cite{GP}.

Our proof of \cref{thm:Gr} employs a certain vector field $\tau$ on $\Grtnn$. The flow defined by $\tau$ contracts all of $\Grtnn$ to a unique fixed point $X_0\in\Grtnn$. We construct a homeomorphism from $\Grtnn$ to a closed ball $B\subset\Grtnn$ centered at $X_0$, by mapping each trajectory in $\Grtnn$ to its intersection with $B$. A feature of our construction is that we do not rely on any cell decomposition of $\Grtnn$.

\subsection{The totally nonnegative part of the unipotent radical of \texorpdfstring{$\GL_n(\RR)$}{GL(n)}}\label{intro:U}
The interest in totally nonnegative spaces from the viewpoint of combinatorial topology dates back at least to Fomin and Shapiro~\cite{FS}. Edelman~\cite{edelman_81} had shown that intervals in the poset formed by the symmetric group $\mathfrak{S}_n$ with Bruhat order are shellable, whence Bj\"orner's results~\cite{Bjo} imply that there exists a regular CW complex homeomorphic to a ball whose face poset is isomorphic to $\mathfrak{S}_n$.  Fomin and Shapiro~\cite{FS} suggested that such a CW complex could be found naturally occurring in the theory of total positivity.

Namely, let $U \subset \GL_n(\RR)$ be the subgroup of all upper-triangular unipotent matrices, and $U_{\ge 0}$ its totally nonnegative part, where all minors are nonnegative. Let $V_{\ge 0}$ denote the link of the identity of $U_{\ge 0}$. The intersection of $V_{\geq 0}$ with the Bruhat stratification of $U$ induces a decomposition of $V_{\geq 0}$ into cells, whose face poset is isomorphic to $\mathfrak{S}_n$. Fomin and Shapiro~\cite[Conjecture~1.10]{FS} conjectured that $V_{\ge 0}$ is a regular CW complex, which was proved by Hersh~\cite{Hersh}. Applying her result to the cell of top dimension implies that $V_{\geq 0}$ is homeomorphic to an $\left(\binom{n}{2}-1\right)$-dimensional closed ball. We give a new proof of this special case to exhibit the wide applicability of our methods. We emphasize that our techniques in their present form are not able to address the other (lower-dimensional) cell closures in $V_{\ge 0}$, which appear in Hersh's result.  In addition, Fomin and Shapiro's conjecture, as well as Hersh's theorem, hold in arbitrary Lie types, while we only consider type $A$.

\subsection{The \cs amplituhedron}\label{intro:amplituhedron}
A robust connection between the totally nonnegative Grassmannian and the physics of scattering amplitudes was developed in~\cite{abcgpt}, which led
Arkani-Hamed and Trnka~\cite{AT} to define topological spaces called \emph{amplituhedra}.  A distinguishing feature that these topological spaces share (conjecturally) with convex polytopes is the existence of a canonical differential form \cite{ABL}.  This brings the analogy between totally nonnegative spaces and polytopes beyond the level of face posets.

Let $k,m,n$ be nonnegative integers with $k+m \leq n$, and $Z$ be a $(k+m) \times n$ matrix whose $(k+m) \times (k+m)$ minors are all positive. 
We regard $Z$ as a linear map $\RR^n\to\RR^{k+m}$, which induces a map $Z_{\Gr}$ on $\Gr(k,n)$ taking the subspace $X$ to the subspace $\{Z(v) : v\in X\}$. 
The \emph{(tree) amplituhedron} $\ampl$ is the image of $\Grtnn$ in $\Gr(k,k+m)$ under the map $Z_{\Gr}$~\cite[Section 4]{AT}. When $k = 1$, the totally nonnegative Grassmannian $\Gr_{\ge 0}(1,n)$ is a simplex in $\mathbb{P}^{n-1}$, and the amplituhedron $\mathcal{A}_{n,1,m}(Z)$ is a {\itshape cyclic polytope} in $\mathbb{P}^m$~\cite{sturmfels_88}. Understanding the topology of amplituhedra, and more generally of \emph{Grassmann polytopes}~\cite{LamCDM} (obtained by relaxing the positivity condition on $Z$), was one of the main motivations of our work.

We now take $m$ to be even, and $Z = Z_0$ such that the rows of $Z_0$ span the unique element of $\Gr_{\ge 0}(k+m,n)$ invariant under the $\Z/n\Z$-cyclic action (cf.~\cite{karp_cyclic_shift}).  We call $\amplZsym$ the {\itshape \cs amplituhedron}.  When $k = 1$ and $m =2$, $\mathcal{A}_{n,1,2}(Z_0)$ is a regular $n$-gon in the plane.  More generally, $\mathcal{A}_{n,1,m}(Z_0)$ is a polytope whose vertices are $n$ regularly spaced points on the {\itshape trigonometric moment curve} in $\mathbb{P}^{m}$.

\begin{thm}\label{thm:ampli}
The \cs amplituhedron $\amplZsym$ is homeomorphic to a $km$-dimensional closed ball.
\end{thm}

It is expected that every amplituhedron is homeomorphic to a closed ball.  The topology of amplituhedra and Grassmann polytopes is not well understood in general; see~\cite{Karp-Williams, Arkani-Hamed-Thomas-Trnka} for recent work. 

\subsection{The compactification of the space of planar electrical networks}\label{intro:E}
Let $\Gamma$ be an electrical network consisting only of resistors, modeled as an undirected graph whose edge weights (conductances) are positive real numbers. The electrical properties of $\Gamma$ are encoded by the response matrix $\Lambda(\Gamma): \RR^n \to \RR^n$, sending a vector of voltages at $n$ distinguished boundary vertices to the vector of currents induced at the same vertices.  The response matrix can be computed using (only) Kirchhoff's law and Ohm's law.  Following Curtis, Ingerman, and Morrow~\cite{CIM} and Colin de Verdi\`ere, Gitler, and Vertigan~\cite{CGV}, we consider the space $\Omega_n$ of response matrices of planar electrical networks: those $\Gamma$ embedded into a disk, with boundary vertices on the boundary of the disk. This space is not compact; a compactification $E_n$ was defined by the third author in~\cite{Lam}. It comes equipped with a natural embedding $\iota: E_n \hookrightarrow \Gr_{\geq 0}(n-1,2n)$. We exploit this embedding to establish the following result.

\begin{thm}\label{thm:En}
The space $E_n$ is homeomorphic to an $\binom{n}{2}$-dimensional closed ball.
\end{thm}

A cell decomposition of $E_n$ was defined in~\cite{Lam}, extending earlier work in~\cite{CIM,CGV}. The face poset of this cell decomposition had been defined and studied by Kenyon~\cite[Section~4.5.2]{Kenyon}. \cref{thm:En} says that the closure of the unique cell of top dimension in $E_n$ is homeomorphic to a closed ball.  In~\cite{Lammatchings}, the third author showed that the face poset of the cell decomposition of $E_n$ is Eulerian, and conjectured that it is shellable. Hersh and Kenyon recently proved this conjecture~\cite{hersh_kenyon}. Bj\"orner's results~\cite{Bjo} therefore imply that this poset is the face poset of some regular CW complex homeomorphic to a ball. We expect that $E_n$ forms such a CW complex, so that the closure of every cell of $E_n$ is homeomorphic to a closed ball. Proving this remains an open problem.

\subsection{Outline}
In \cref{sec:contracting}, we prove a topological lemma (\cref{lem:top}) which essentially states that if one can find a \emph{contractive flow} (defined below) on a submanifold of $\RR^N$, then its closure is homeomorphic to a ball. We use \cref{lem:top} in \cref{sec:Gr,sec:U,sec:amplituhedron,sec:E} to show  that the four spaces discussed in \cref{intro:Gr,intro:U,intro:amplituhedron,intro:E} are homeomorphic to closed balls. To do so, in each case we consider a natural flow on the underlying space, and show that it satisfies the contractive property by introducing novel coordinates on the space.\\

\noindent {\bfseries Acknowledgements.} We thank Patricia Hersh and Lauren Williams for helpful comments, and anonymous referees for many suggestions leading to improvements in the exposition.

\section{Contractive flows}\label{sec:contracting}

\noindent In this section we prove \cref{lem:top}, which we will repeatedly use in establishing our main theorems. Consider a real normed vector space $(\RR^N,\|\cdot\|)$. Thus for each $r>0$, the closed ball $B_r^N:=\{p\in\RR^N: \|p\|\leq r\}$ of radius $r$ is a compact convex body in $\RR^N$ whose interior contains the origin. We denote its boundary by $\partial B_r^N$, which is the sphere of radius $r$.

\begin{defn}\label{dfn:contract}
We say that a map $\flow:\RR\times \RR^N\to\RR^N$ is a \emph{contractive flow} if the following conditions are satisfied:
\begin{enumerate}[(1)]
\item\label{item:continuous} the map $\flow$ is continuous;
\item\label{item:action} for all $p\in\RR^N$ and $t_1,t_2\in\RR$, we have $\act 0 p=p$ and $\act{t_1+t_2}{p}=\act{t_1}{\act{t_2}{p}}$; and
\item\label{eq:contract} for all $p\neq 0$ and $t > 0$, we have $\|\act{t}{p}\| < \|p\|$.
\end{enumerate}
\end{defn}

The condition~\eqref{item:action} says that $\flow$ induces a group action of $(\RR, +)$ on $\RR^N$. In particular, $\act{t}{p} = q$ is equivalent to $\act{-t}{q} = p$, so~\eqref{eq:contract} implies that if $t\neq 0$ and $\act{t}{p} = {p}$, then $p = 0$. The converse to this statement is given below in \cref{lem:limits_act}\eqref{act:1}.
\begin{lem}\label{lem:limits_act}
Let $\flow:\RR\times \RR^N\to\RR^N$ be a contractive flow.
\begin{enumerate}[(i)]
\item\label{act:1}
We have $\act{t}{0} = 0$ for all $t\in\RR$.
\item\label{act:2}
Let $p\neq 0$. Then the function $t\mapsto\|\act{t}{p}\|$ is strictly decreasing on $(-\infty,\infty)$.
\item\label{act:3}
Let $p\neq 0$. Then $\displaystyle\lim_{t\to\infty}\|\act t p\|=0$ and $\displaystyle\lim_{t\to-\infty} \|\act t p\|=\infty$.
\end{enumerate}
\end{lem}

\begin{proof}
\eqref{act:1}
By~\eqref{item:continuous}, the function $s\mapsto\|\act{s}{0}\|$ is continuous on $\RR$, and it equals $0$ when $s=0$. If $\act{t}{0}\neq 0$ for some $t > 0$, then $0 < \|\act{s}{0}\| < \|\act{t}{0}\|$ for some $s\in (0,t)$, which contradicts~\eqref{eq:contract} applied to $p = \act{s}{0}$ and $t-s$. Therefore $\act{t}{0} = 0$ for all $t\geq0$. By~\eqref{item:action}, for $t\ge 0$ we have  $0 = \act{0}{0} = \act{-t}{\act{t}{0}}=\act{-t}0$, and so $\act{-t}{0} = 0$ as well.

\eqref{act:2} This follows from~\eqref{eq:contract} and the fact that $\flow$ induces a group action of $\RR$ on $\RR^N$, once we know that $\act{t}{p}$ is never $0$. But if $\act{t}{p} = 0$ then $\act{-t}{0} = p$, which contradicts part~\eqref{act:1}.

\eqref{act:3} Let $r_1(p)$ and $r_2(p)$ denote the respective limits. By part~\eqref{act:2}, both limits exist, where $r_1(p)\in [0,\infty)$ and $r_2(p)\in (0,\infty]$. By compactness, there exists a point $q\in\RR^N$ with $\|q\| = r_1(p)$, along with an unbounded increasing sequence $t_1,t_2,\dots$ in $\mathbb{R}$ satisfying $\lim_{i\to\infty} \act{t_i}{p}=q$. If $r_1(p) > 0$, then using~\eqref{item:continuous}--\eqref{eq:contract}, we find
\begin{align*}
r_1(p)=\|q\|>\|\act{1}{q}\|=\lim_{i\to\infty}\|\act{1}{\act{t_i}{p}}\|=\lim_{i\to\infty}\|\act{1+t_i}{p}\|=r_1(p),
\end{align*}
a contradiction. Thus $r_1(p)=0$. Similarly, we get $r_2(p)=\infty$.\footnote{We thank an anonymous referee for suggesting this simpler argument.}
\end{proof}

For $K\subset \RR^N$ and $t\in\RR$, we let $\act{t}{K}$ denote $\{\act{t}{p}: p\in K\}$.
\begin{lem}\label{lem:top}
Let $\openball \subset \RR^N$ be a smooth embedded submanifold of dimension $d \leq N$, and $\flow:\RR\times \RR^N\to\RR^N$ a contractive flow. Suppose that $\openball$ is bounded and satisfies the condition
\begin{align}\label{eq:invariant}
\act{t}{\closedball} \subset \openball \quad \text{ for $t > 0$}.
\end{align}
Then the closure $\closedball$ is homeomorphic to a closed ball of dimension $d$, and $\closedball \setminus \openball$ is homeomorphic to a sphere of dimension $d-1$.
\end{lem}
Note that any open subset of $\RR^N$ is a smooth embedded submanifold of dimension $N$.

\begin{proof}
Since $\openball$ is bounded, its closure $\closedball$ is compact. By \cref{lem:limits_act}\eqref{act:3} and \eqref{eq:invariant} we have $0 \in \closedball$, and therefore $0\in\openball$. Because $\openball$ is smoothly embedded, we can take $r>0$ sufficiently small so that $B := B^N_r \cap \openball$ is homeomorphic to a closed ball of dimension $d$. We let $\partial B$ denote $(\partial B^N_r)\cap\openball$, which is a $(d-1)$-dimensional sphere.

For any $p\in\RR^N\setminus\{0\}$, consider the curve $t\mapsto \act t p$ starting at $p$ and defined for all $t\in\RR$. By \cref{lem:limits_act}\eqref{act:2}, this curve intersects the sphere $\partial B^N_r$ for a unique $t\in\RR$, which we denote by $t_r(p)$. 
Also, for $p \in \closedball \setminus \{0\}$, define $t_\partial(p)\in(-\infty,0]$ as follows.  Let $T(p):=\{t\in\RR : \act{t}{p}\in \closedball\}$. We have $0\in T(p)$, and $T(p)$ is bounded from below by \cref{lem:limits_act}\eqref{act:3} because $\closedball$ is bounded.  By \eqref{eq:invariant}, if $t\in T(p)$ then $[t,\infty)\subset T(p)$.  Also, $T(p)$ is closed since it is the preimage of $\closedball$ under the continuous map $t\mapsto \act{t}{p}$.  It follows that $T(p) = [t_\partial(p),\infty)$ for some $t_\partial(p)\in(-\infty,0]$.
\begin{claim}
The functions $t_r$ and $t_\partial$ are continuous on $\closedball\setminus\{0\}$.
\end{claim}

\begin{claimpf}
First we prove that $t_r$ is continuous on $\RR^N\setminus\{0\}$. It suffices to show that the preimage of any open interval $I\subset\RR$ is open. To this end, let $q\in t_r^{-1}(I)$. Take $t_1, t_2\in I$ with $t_1 < t_r(q) < t_2$. By \cref{lem:limits_act}\eqref{act:2}, we have $\|\act{t_1}{q}\| > r > \|\act{t_2}{q}\|$. Note that the map $\gamma_1 : \RR^N \to \RR^N, p \mapsto \act{t_1}p$ is continuous and $\RR^N \setminus B_r^N$ is open, so $\gamma_1^{-1}(\RR^N \setminus B_r^N)$ is an open neighborhood of $q$. Similarly, defining $\gamma_2 : \RR^N \to \RR^N, p\mapsto \act{t_2}p$, we have that $\gamma_2^{-1}(\interior(B^N_r))$ is an open neighborhood of $q$. Therefore $\gamma_1^{-1}(\RR^N \setminus B_r^N) \cap \gamma_2^{-1}(\interior(B^N_r))$ is an open neighborhood of $q$, whose image under $t_r$ is contained in $(t_1, t_2)\subset I$. This shows that $t_r$ is continuous on $\RR^N\setminus\{0\}$.

Next, let us define
\begin{align}\label{eq:M}
\affinespan:= \{\act t p : (t,p) \in \RR \times \closedball\}.
\end{align}
The map $b: \RR \times \partial B \to \affinespan \setminus \{0\}$ defined by $(t,p) \mapsto \act t p$ is a continuous bijection. (Recall that $\partial B = (\partial B^N_r)\cap\openball$.) Its inverse $p\mapsto (-t_r(p),\act{t_r(p)}p)$ is continuous as well. Therefore $b$ is a homeomorphism. We claim that $\openball$ is relatively open in $\affinespan$. Indeed, since $\openball\setminus\{0\}$ is a submanifold of $\affinespan\setminus \{0\}$ of the same dimension $d$, we deduce that $\openball\setminus\{0\}$ is an open subset of $\affinespan\setminus \{0\}$. Also, $\openball$ contains the neighborhood $\interior(B^N_r) \cap R$ of $0$ in $R$. Thus $\openball$ is an open subset of $\affinespan$.

We now prove that the map $t_\partial:\closedball\setminus\{0\}\to\RR$ is continuous, by a very similar argument. Let $I\subset \RR$ be an open interval and consider a point $q\in t_\partial^{-1}(I)$. Take $t_1, t_2\in I$ with $t_1 < t_\partial(q) < t_2$. By the definition of $t_\partial$, we have $\act{t_1}q\in \affinespan\setminus \closedball$. By~\eqref{eq:invariant}, we have  $\act{t_2}q\in\openball$. Note that the map $\gamma_1 : \affinespan \to \affinespan, p \mapsto \act{t_1}p$ is continuous and $\affinespan\setminus \closedball$ is open in $\affinespan$, so $\gamma_1^{-1}(\affinespan\setminus \closedball)$ is an open neighborhood of $q$ in $\affinespan$. Similarly, defining $\gamma_2 : \affinespan\to\affinespan, p\mapsto \act{t_2}p$, we have that $\gamma_2^{-1}(\openball)$ is an open neighborhood of $q$ in $\affinespan$. Therefore $\gamma_1^{-1}(\affinespan\setminus \closedball)\cap \gamma_2^{-1}(\openball)\cap\closedball$ is an open neighborhood of $q$ in $\closedball$, whose image under $t_\partial$ is contained inside $(t_1,t_2)\subset I$. This finishes the proof of the claim.
\end{claimpf}

Define the maps $\alpha:\closedball \to B$ and $\beta: B\to\closedball$ by
\[
\alpha(p) := \act{t_r(p) - t_\partial(p)}p,\quad \beta(p) := \act{t_\partial(p) - t_r(p)}p
\]
for $p\neq 0$, and $\alpha(0) := 0$, $\beta(0) := 0$. Let us verify that $\alpha$ sends $\closedball$ inside $B$ and $\beta$ sends $B$ inside $\closedball$. If $p\in\closedball\setminus\{0\}$, then $\act{t_r(p)}p\in B$ and $t_\partial(p)\leq 0$, whence the contractive property~\eqref{eq:contract} implies $\alpha(p)=\act{-t_\partial(p)}{\act{t_r(p)}p}\in B$. Similarly, if $p\in B\setminus\{0\}$, then $\act{t_\partial(p)}{p}\in\closedball$ and $t_r(p)\leq 0$, whence~\eqref{eq:invariant} implies $\beta(p) = \act{-t_r(p)}{\act{t_\partial(p)}{p}}\in\closedball$.

Now we check that $\alpha$ and $\beta$ are inverse maps. For any $p\in\closedball$ and $\Delta t\in\RR$ such that $\act{\Delta t}p\in \closedball$, we have
\[
t_r(\act{\Delta t}p)=t_r(p)-\Delta t,\quad t_\partial(\act{\Delta t}p)=t_\partial(p)-\Delta t.
\]
Taking $\Delta t:=t_\partial(p)-t_r(p)$, we find
\[
\alpha(\beta(p)) = \alpha(\act{\Delta t}{p}) = \act{t_r(\act{\Delta t}{p}) - t_\partial(\act{\Delta t}{p})}{\act{\Delta t}p} = \act{-\Delta t}{\act{\Delta t}p} = p.
\]
We can similarly verify that $\beta(\alpha(p))=p$, by instead taking $\Delta t := t_r(p)-t_\partial(p)$.

By the claim, $t_r$ and $t_\partial$ are continuous on $\closedball\setminus\{0\}$, so $\alpha$ is continuous everywhere except possibly at $0$. Also, $t_r(p)>t_\partial(p)$ for all $p\in\closedball\setminus \{0\}$, so $\alpha$ is continuous at $0$ by \cref{lem:limits_act}\eqref{act:2}. Thus $\alpha$ is a continuous bijection from a compact space to a Hausdorff space, so it is a homeomorphism. This shows that $\closedball$ is homeomorphic to a closed $d$-dimensional ball.

It remains to prove that $\closedball\setminus\openball$ is homeomorphic to a $(d-1)$-dimensional sphere. We claim that $\alpha$ restricts to a homeomorphism from $\closedball\setminus\openball$ to $\partial B$. We need to check that $\alpha$ sends $\closedball\setminus\openball$ inside $\partial B$, and $\beta$ sends $\partial B$ inside $\closedball\setminus\openball$. To this end, let $p\in\closedball\setminus\openball$. By condition~\eqref{item:action}, we have $p = \act{-t_\partial(p)}{\act{t_\partial(p)}{p}}$. Hence if $t_\partial(p) < 0$, then~\eqref{eq:invariant} implies $p\in Q$, a contradiction. Therefore $t_\partial(p) = 0$, and $\alpha(p) = f(t_r(p),p)\in\partial B$. Now let $q\in\partial B$. We have $t_r(q) = 0$, so $\beta(q) = \act{t_\partial(q)}{q}$. If $\beta(q)\in Q$, then $\act{t_\partial(q)-t}{q}\in Q$ for $t > 0$ sufficiently small (as $Q$ is open in $R$ from~\eqref{eq:M}), contradicting the definition of $t_\partial(q)$. Thus $\beta(q)\in\closedball\setminus\openball$.
\end{proof}

\section{The totally nonnegative Grassmannian}\label{sec:Gr}

\noindent Let $\Gr(k,n)$ denote the \emph{real Grassmannian}, the space of all $k$-dimensional subspaces of $\RR^n$. We set $[n]:=\{1,\dots,n\}$, and let $\binom{[n]}{k}$ denote the set of $k$-element subsets of $[n]$. For $X\in\Gr(k,n)$, we denote by $(\Delta_I(X))_{I\in{\binom{[n]}{k}}}\in\RP^{\binom{n}{k}-1}$ the \emph{Pl\"ucker coordinates} of $X$: $\Delta_I(X)$ is the $k\times k$ minor of $X$ (viewed as a $k\times n$ matrix modulo row operations) with column set $I$.

Recall that $\Gr_{\ge 0}(k,n)$ is the subset of $\Gr(k,n)$ where all Pl\"{u}cker coordinates are nonnegative (up to a common scalar). We also define the {\itshape totally positive Grassmannian} $\Gr_{>0}(k,n)$ as the subset of $\Gr_{\ge 0}(k,n)$ where all Pl\"ucker coordinates are positive.

\subsection{Global coordinates for \texorpdfstring{$\Gr_{\ge 0}(k,n)$}{the totally nonnegative Grassmannian}}\label{sec:glob-coord}

For each $k$ and $n$, we introduce several distinguished linear operators on $\RR^n$. Define the {\itshape left cyclic shift} $\shift\in\gl_n(\RR) = \End(\RR^n)$ by $\shift(v_1, \dots, v_n) := (v_2, \dots, v_n, (-1)^{k-1}v_1)$. The sign $(-1)^{k-1}$ can be explained as follows: if we pretend that $\shift$ is an element of $\GL_n(\RR)$, then the action of $\shift$ on $\Gr(k,n)$ preserves $\Grtnn$ (it acts on Pl\"{u}cker coordinates by rotating the index set $[n]$). 

Note that the transpose $\shift^T$ of $\shift$ is the {\itshape right cyclic shift} given by $S^T(v_1, \dots, v_n) = ((-1)^{k-1}v_n, v_1, \dots, v_{n-1})$. Let $\tau := \shift+\shift^T\in\End(\RR^n)$. We endow $\RR^n$ with the standard inner product, so that $\tau$ (being symmetric) has an orthogonal basis of eigenvectors $\eigtau_1,\dots,\eigtau_n\in\RR^n$ corresponding to real eigenvalues $\l_1\geq\dots\geq \l_n$. Let $X_0\in\Gr(k,n)$ be the linear span of $u_1, \dots, u_k$. The following lemma implies that $X_0$ is totally positive and does not depend on the choice of eigenvectors $u_1, \dots, u_n$.
\begin{lem}\label{tau_properties}~
\begin{enumerate}[(i)] 
\item\label{tau_eigenvalues}
The eigenvalues $\l_1\geq\dots\geq \l_n$ are given as follows, depending on the parity of $k$:
\begin{itemize}
\item if $k$ is even, $\l_1 = \l_2 = 2\cos(\frac{\pi}{n})$, $\l_3 = \l_4 = 2\cos(\frac{3\pi}{n})$, $\l_5 = \l_6 = 2\cos(\frac{5\pi}{n})$, \dots;
\item if $k$ is odd, $\l_1 = 1$, $\l_2 = \l_3 = 2\cos(\frac{2\pi}{n})$, $\l_4 = \l_5 = 2\cos(\frac{4\pi}{n})$, \dots. 
\end{itemize}
In either case, we have
\begin{align*}
\lambda_k = 2\cos\!\left(\textstyle\frac{k-1}{n}\pi\right) > 2\cos\!\left(\textstyle\frac{k+1}{n}\pi\right) = \lambda_{k+1}.
\end{align*}
\item\label{prop:X_0}\cite{scott}
The Pl\"{u}cker coordinates of $X_0$ are given by
\[
\Delta_I(X_0) = \prod_{i,j\in I,\; i < j}\sin\!\left(\textstyle\frac{j-i}{n}\pi\right) > 0 \quad \text{ for all $I\in\textstyle\binom{[n]}{k}$}.
\]
\end{enumerate}
\end{lem}
For an example in the case of $\Gr(2,4)$, see \cref{eg_Gr24}. (We remark that in the example, the Pl\"{u}cker coordinates of $X_0$ are scaled by a factor of $2$ compared to the formula above.)
\begin{proof}
In this proof, we work over $\CC$. Let $\zeta\in\CC$ be an $n$th root of $(-1)^{k-1}$. There are $n$ such values of $\zeta$, each of the form $\zeta = e^{i\pi m/n}$ for some integer $m$ congruent to $k-1$ modulo $2$. Let $z_m := (1, \zeta, \zeta^2, \dots, \zeta^{n-1})\in\mathbb{C}^n$. We have $S(z_m) = \zeta z_m$ and $S^T(z_m) = \zeta^{-1} z_m$, so
\begin{align}\label{eigenequation}
\tau(z_m) = (\zeta + \zeta^{-1})z_m = 2\cos(\textstyle\frac{\pi m}{n})z_m.
\end{align}
The $n$ distinct $z_m$'s are linearly independent (they form an $n\times n$ Vandermonde matrix with nonzero determinant), so they give a basis of $\CC^n$ of eigenvectors of $\tau$.

We deduce part~\eqref{tau_eigenvalues} from~\eqref{eigenequation}. For part~\eqref{prop:X_0}, we apply Vandermonde's determinantal identity, following an argument outlined by Scott~\cite{scott}. That is, by~\eqref{eigenequation}, the $\CC$-linear span of $u_1, \dots, u_k$ is the same as the span of $z_{-k+1}, z_{-k+3}, z_{-k+5}, \dots, z_{k-1}$. Let $M$ be the matrix whose rows are $z_{-k+1}, z_{-k+3}, z_{-k+5}, \dots, z_{k-1}$, i.e.,
\[
M_{r,j} = e^{i\pi (-k-1+2r)(j-1)/n} \quad \text{ for $1\leq r\leq k$ and $1\leq j\leq n$}.
\]
Then the Pl\"{u}cker coordinates of $X_0$ are the $k\times k$ minors of $M$ (up to a common nonzero complex scalar), which can be computed explicitly by Vandermonde's identity after appropriately rescaling the columns. We refer the reader to~\cite[Proposition 2.5]{karp_cyclic_shift} for details.
\end{proof}

Denote by $\mat$ the vector space of real $k\times (n-k)$ matrices. Define a map $\phi:\mat\to\Gr(k,n)$ by
\begin{align}\label{not:embedding}
\phi(A) := \spn(\eigtau_i + \textstyle\sum_{j=1}^{n-k}A_{i,j}\eigtau_{k+j} : 1\leq i\leq k).
\end{align}
In other words, the entries of $A$ are the usual coordinates on the big Schubert cell of $\Gr(k,n)$ with respect to the basis $u_1, \dots, u_n$ of $\RR^n$, this Schubert cell being
\[
\phi(\mat) =  \{X\in\Gr(k,n) : X\cap\spn(u_{k+1}, \dots, u_n)= 0\}.
\]
In particular, $\phi$ is a smooth embedding, and it sends the zero matrix to $X_0$. For an example in the case of $\Gr(2,4)$, see \cref{eg_Gr24}.

\begin{prop}\label{embedding_properties}
The image $\phi(\mat)$ contains $\Grtnn$.
\end{prop}

\begin{proof}
Let $X\in\Grtnn$ be a totally nonnegative subspace. We need to show that $X\cap\spn(u_{k+1}, \dots, u_n)=0$. Suppose otherwise that there exists a nonzero vector $v$ in this intersection. Extend $v$ to a basis of $X$, and write this basis as the rows of a $k\times n$ matrix $M$. Because $X$ is totally nonnegative, the nonzero $k\times k$ minors of $M$ all have the same sign (and at least one minor is nonzero, since $M$ has rank $k$). Also let $M_0$ be the $k\times n$ matrix with rows $u_1, \dots, u_k$. By \cref{tau_properties}\eqref{prop:X_0}, all $k\times k$ minors of $M_0$ are nonzero and have the same sign. The vectors $u_1, \dots, u_n$ are orthogonal, so $v$ is orthogonal to the rows of $M_0$. Hence the first column of $M_0M^T$ is zero, and we obtain $\det(M_0M^T) = 0$. On the other hand, the Cauchy--Binet identity implies
\[
\det(M_0M^T) = \sum_{I\in\binom{[n]}{k}}\det((M_0)_I)\det(M_I),
\]
where $A_I$ denotes the matrix $A$ restricted to the columns $I$. Each summand has the same sign and at least one summand is nonzero, contradicting $\det(M_0M^T) = 0$.
\end{proof}

We have shown that the restriction of $\phi^{-1}$ to $\Grtnn$ yields an embedding
\[
\Grtnn\hookrightarrow\Mat(k,n-k)\simeq \RR^{k(n-k)}
\]
whose restriction to $\Grtp$ is smooth.

\subsection{Flows on \texorpdfstring{$\Gr(k,n)$}{Gr(k,n)}}\label{sec:cyclic-shift-s}

For $g\in\GL_n(\RR)$, we let $g$ act on $\Gr(k,n)$ by taking the subspace $X$ to $g\cdot X := \{g(v) : v\in X\}$. We let $1\in\GL_n(\RR)$ denote the identity matrix, and for $x\in\gl_n(\RR)$ we let $\exp(x) := \sum_{j=0}^\infty\frac{x^j}{j!}\in\GL_n(\RR)$ denote the matrix exponential of $x$.

We examine the action of $\expts$ and $\expttau$ on $\Gr(k,n)$.
\begin{lem}\label{exp(ts)}
For $X\in\Grtnn$ and $t > 0$, we have $\expts\cdot X\in\Grtp$.
\end{lem}

\begin{proof}
We claim that it suffices to prove the following two facts:
\begin{enumerate}[(i)]
\item\label{exp(ts)_1} for $X\in\Grtnn$ and $t \ge 0$, we have $\expts\cdot X\in\Grtnn$; and
\item\label{exp(ts)_2} for $X\in\Grtnn\setminus\Grtp$, we have $\expts\cdot X\notin\Grtnn$ for all $t < 0$ sufficiently close to zero.
\end{enumerate}
To see why this is sufficient, let $X\in\Grtnn$ and $t > 0$. By part~\eqref{exp(ts)_1}, we have $\expts\cdot X\in\Grtnn$, so we just need to show that $\expts\cdot X\in\Grtp$. Suppose otherwise that $\expts\cdot X\notin\Grtp$. Then applying part~\eqref{exp(ts)_2} to $\expts\cdot X$, we get that $\exp((t+t')\shift)\cdot X\notin\Grtnn$ for $t' < 0$ sufficiently close to zero. But by part~\eqref{exp(ts)_1}, we know that $\exp((t+t')S)\cdot X\in\Grtnn$ for all $t'$ in the interval $[-t,0]$. This is a contradiction.

Now we prove parts~\eqref{exp(ts)_1} and~\eqref{exp(ts)_2}. We will make use of the operator $1+t\shift$, which belongs to $\GL_n(\RR)$ for $|t|<1$. Note that if $[M_1 \mid \cdots \mid M_n]$ is a $k\times n$ matrix representing $X$, then a $k\times n$ matrix representing $(1+t\shift)\cdot X$ is
\[
M' = [M_1 + tM_2 \mid M_2 + tM_3 \mid \cdots \mid M_{n-1} + tM_n \mid M_n + (-1)^{k-1}tM_1].
\]
We can evaluate the $k\times k$ minors of $M'$ using multilinearity of the determinant. We obtain
\begin{align}\label{eq:1+tS}
\Delta_I((1+t\shift)\cdot X) = \sum_{\epsilon\in\{0,1\}^k}t^{\epsilon_1 + \dots + \epsilon_k}\Delta_{\{i_1+\epsilon_1, \dots, i_k+\epsilon_k\}}(X) \quad \text{ for } I = \{i_1, \dots, i_k\}\subset [n],
\end{align}
where $i_1+\epsilon_1, \dots, i_k+\epsilon_k$ are taken modulo $n$. Therefore $(1+t\shift)\cdot X\in\Grtnn$ for $X\in\Grtnn$ and $t\in [0,1)$. Since $\expts = \lim_{j\to\infty}\left(1+\frac{tS}{j}\right)^j$ and $\Grtnn$ is closed, we obtain $\expts\cdot X\in\Grtnn$ for $t \ge 0$. This proves part~\eqref{exp(ts)_1}.

To prove part~\eqref{exp(ts)_2}, first note that $\expts = 1 + t\shift + O(t^2)$. By~\eqref{eq:1+tS}, we have
\begin{align}\label{expts_Pluckers}
\Delta_I(\expts\cdot X) = \Delta_I(X) + t\sum_{I'}\Delta_{I'}(X) + O(t^2)\quad\text{ for }I\in\textstyle\binom{[n]}{k},
\end{align}\par\vspace*{-4pt}\noindent
where the sum is over all $I'\in\binom{[n]}{k}$ obtained from $I$ by increasing exactly one element by $1$ modulo $n$. If we can find such $I$ and $I'$ with $\Delta_I(X) = 0$ and $\Delta_{I'}(X) > 0$, then $\Delta_I(\expts\cdot X) < 0$ for all $t < 0$ sufficiently close to zero, thereby proving part~\eqref{exp(ts)_2}. In order to do this, we introduce the directed graph $D$ with vertex set $\binom{[n]}{k}$, where $J\to J'$ is an edge of $D$ if and only if we can obtain $J'$ from $J$ by increasing exactly one element by $1$ modulo $n$. Note that for any two vertices $K$ and $K'$ of $D$, there exists a directed path from $K$ to $K'$:
\begin{enumerate}[\hspace*{24pt}$\bullet$]
\item we can get from $[k]$ to any $\{i_1 < \dots < i_k\}$ by shifting $k$ to $i_k$, $k-1$ to $i_{k-1}$, etc.;
\item similarly, we can get from any $\{i_1 < \dots < i_k\}$ to $\{n-k+1, n-k+2, \dots, n\}$;
\item we can get from $\{n-k+1, \dots, n\}$ to $[k]$ by shifting $n$ to $k$, $n-1$ to $k-1$, etc.
\end{enumerate}
Now take $K,K'\in\binom{[n]}{k}$ with $\Delta_K(X) = 0$ and $\Delta_{K'}(X) > 0$, and consider a directed path from $K$ to $K'$. It goes through an edge $I\to I'$ with $\Delta_I(X) = 0$ and $\Delta_{I'}(X) > 0$, as desired.
\end{proof}

Now we consider $\expttau=\exp(t(\shift+\shift^T))$. Recall that $S$ and $S^T$ are the left and right cyclic shift maps, so by symmetry \cref{exp(ts)} holds with $S$ replaced by $S^T$. Also, $S$ and $S^T$ commute, so $\expttau = \exp(t\shift)\exp(t\shift^T)$. We obtain the following.
\begin{cor}\label{cor:expttau_preserves_Grtp}
For $X\in\Grtnn$ and $t > 0$, we have $\expttau\cdot X\in\Grtp$.
\end{cor}

Let us see how $\expttau$ acts on matrices $A\in\mat$. Note that $\tau(\eigtau_i) = \l_i \eigtau_i$ for $1\leq i\leq n$, so $\expttau(\eigtau_i) = e^{t\l_i}\eigtau_i$. Therefore $\expttau$ acts on the basis of $\phi(A)$ in~\eqref{not:embedding} by
\[
\expttau(\eigtau_i + \textstyle\sum_{j=1}^{n-k}A_{i,j}\eigtau_{k+j}) =  e^{t\l_i}(\eigtau_i + \textstyle\sum_{j=1}^{n-k} e^{t(\l_{k+j}-\l_i)}A_{i,j}\eigtau_{k+j})
\]
for all $1\leq i\leq k$. Thus $\expttau\cdot\phi(A) = \phi(\act t A)$, where by definition $\act t A\in\mat$ is the matrix with entries
\begin{align}\label{eq:expts_dfn}
(\act t A)_{i,j} := e^{t(\l_{k+j}-\l_i)}A_{i,j}\quad\text{ for $1\leq i\leq k$ and $1\leq j\leq n-k$}.
\end{align}

\subsection{Proof of \texorpdfstring{\cref{thm:Gr}}{Theorem 1.1}}\label{sec:proof}
Consider the map $\flow:\RR\times \mat\to\mat$ defined by~\eqref{eq:expts_dfn}. We claim that $\flow$ is a contractive flow on $\mat$ equipped with the Euclidean norm
\[
\|A\|^2 = \sum_{i=1}^k \sum_{j=1}^{n-k}A_{i,j}^2.
\]
Indeed, parts~\eqref{item:continuous} and~\eqref{item:action} of \cref{dfn:contract} hold for $\flow$. To see that part~\eqref{eq:contract} holds, note that for any $1 \le i \le k$ and $1 \le j \le n-k$ with $A_{i,j}\neq 0$, we have
\[
|(\act t A)_{i,j}| = |e^{t(\l_{k+j}-\l_i)}A_{i,j}| = e^{t(\l_{k+j}-\l_i)}|A_{i,j}| < |A_{i,j}| \quad \text{ for $t > 0$},
\]
using the fact that $\l_i \ge \l_k > \l_{k+1} \ge \l_{k+j}$ from \cref{tau_properties}\eqref{tau_eigenvalues}. Therefore $\|\act t A\| < \|A\|$ if $A\neq 0$, verifying part~\eqref{eq:contract}.

Let us now apply \cref{lem:top} with $\RR^N=\mat$ and $\openball=\phi^{-1}(\Grtp)$. We need to know that $\Grtnn$ is the closure of $\Grtp$. This was proved by Postnikov~\cite[Section 17]{Pos}; it also follows directly from \cref{cor:expttau_preserves_Grtp}, since we can express any $X\in\Grtnn$ as a limit of totally positive subspaces:
\[
X = \lim_{t\to 0+}\expttau\cdot X.
\]
Therefore $\closedball=\phi^{-1}(\Grtnn)$. Moreover, $\Grtnn$ is closed inside the compact space $\RP^{\binom{n}{k}-1}$, and is therefore also compact.\label{closure} So, $\overline{Q}$ is compact (and hence bounded). Finally, the property~\eqref{eq:invariant} in this case is precisely \cref{cor:expttau_preserves_Grtp}. We have verified all the hypotheses of \cref{lem:top}, and conclude that $\overline{Q}$ (and also $\Grtnn$) is homeomorphic to a $k(n-k)$-dimensional closed ball.\hfill\qedsymbol

\subsection{Related work}\label{sec:remarks}

Lusztig~\cite[Section 4]{LusIntro} used a flow similar to $\expttau$ to show that $(G/P)_{\geq 0}$ is contractible. Our flow can be thought of as an affine (or loop group) analogue of his flow, and is closely related to the whirl matrices of~\cite{LP1}. We also remark that Ayala, Kliemann, and San Martin~\cite{AKSM} used the language of control theory to give an alternative development in type $A$ of Lusztig's theory of total positivity. In that context, $\exp(t\tau)$ ($t > 0$) lies in the interior of the {\itshape compression semigroup} of $\Grtnn$, and $X_0$ is its {\itshape attractor}.

Marsh and Rietsch defined and studied a {\itshape superpotential} on the Grassmannian in the context of mirror symmetry \cite[Section 6]{marsh_rietsch}. It follows from results of Rietsch~\cite{rietsch_08} (see~\cite[Corollary 3.12]{karp_cyclic_shift}) that $X_0$ is, rather surprisingly, also the unique totally nonnegative critical point of the $q=1$ specialization of the superpotential. However, the superpotential is not defined on the boundary of $\Grtnn$. The precise relationship between $\tau$ and the gradient flow of the superpotential remains mysterious.

\subsection{Example: the case \texorpdfstring{$\Gr(2,4)$}{Gr(2,4)}}\label{eg_Gr24}

The matrix $\tau = S + S^T\in\gl_4(\RR)$ and an orthogonal basis of real eigenvectors $u_1, u_2, u_3, u_4$ are
\begin{align*}
\tau = \begin{bmatrix}
0 & 1 & 0 & -1 \\
1 & 0 & 1 & 0 \\
0 & 1 & 0 & 1 \\
-1 & 0 & 1 & 0
\end{bmatrix},\qquad
\begin{aligned}
& u_1 = (0, 1, \sqrt{2}, 1),\quad && \lambda_1 = \sqrt{2}, \\
& u_2 = (-\sqrt{2}, -1, 0, 1),\quad && \lambda_2 = \sqrt{2}, \\
& u_3 = (\sqrt{2}, -1, 0, 1),\quad && \lambda_3 = -\sqrt{2}, \\
& u_4 = (0, 1, -\sqrt{2}, 1),\quad && \lambda_4 = -\sqrt{2}.
\end{aligned}
\end{align*}
The embedding $\phi:\Mat(2,2)\hookrightarrow\Gr(2,4)$ sends the matrix $A = \scalebox{0.86}{$\begin{bmatrix}a & b \\ c & d\end{bmatrix}$}$ to
\begin{align*}
\phi(A) = X = \begin{bmatrix}
u_1 + au_3 + bu_4 \\
u_2 + cu_3 + du_4
\end{bmatrix} = 
\begin{bmatrix}\sqrt{2}\hspace*{1.2pt}a & 1-a+b & \sqrt{2}-\sqrt{2}\hspace*{1.2pt}b & 1+a+b \\
-\sqrt{2} + \sqrt{2}\hspace*{1.2pt}c & -1-c+d & -\sqrt{2}\hspace*{1.2pt}d & 1+c+d
\end{bmatrix}.
\end{align*}
Above we are identifying $X\in\Gr(2,4)$ with a $2\times 4$ matrix whose rows form a basis of $X$. In terms of Pl\"{u}cker coordinates $\Delta_{ij} = \Delta_{\{i,j\}}(X)$, the map $\phi$ is given by
\begin{align}\label{Gr24_Plueckers}
\begin{aligned}
\Delta_{12} &= \sqrt{2}(1 - 2a + b - c + ad - bc), \\[-2pt]
\Delta_{23} &= \sqrt{2}(1 - 2d - b + c + ad - bc), \\[-2pt]
\Delta_{34} &= \sqrt{2}(1 + 2d - b + c + ad - bc), \\[-2pt]
\Delta_{14} &= \sqrt{2}(1 + 2a + b - c + ad - bc),
\end{aligned}\qquad
\begin{aligned}
\Delta_{13} &= 2(1 - b - c - ad + bc), \\[-2pt]
\Delta_{24} &= 2(1 + b + c - ad + bc),
\end{aligned}
\end{align}
and its inverse is given by
\begin{gather*}
\begin{aligned}
&a = (2\Delta_{14} - 2\Delta_{12})/\delta,\quad && b = (\Delta_{12} - \Delta_{23} - \Delta_{34} + \Delta_{14} - \sqrt{2}\hspace*{1.2pt}\Delta_{13} + \sqrt{2}\hspace*{1.2pt}\Delta_{24})/\delta, \\
&d = (2\Delta_{34} - 2\Delta_{23})/\delta,\quad && c = (-\Delta_{12} + \Delta_{23} + \Delta_{34} - \Delta_{14} - \sqrt{2}\hspace*{1.2pt}\Delta_{13} + \sqrt{2}\hspace*{1.2pt}\Delta_{24})/\delta,
\end{aligned}\\
\text{where }\;\delta = \Delta_{12} + \Delta_{23} + \Delta_{34} + \Delta_{14} + \sqrt{2}\hspace*{1.2pt}\Delta_{13} + \sqrt{2}\hspace*{1.2pt}\Delta_{24}.\quad
\end{gather*}

The point $X_0 = \phi(0) = \spn(u_1, u_2)\in\Gr_{>0}(2,4)$ has Pl\"{u}cker coordinates
\[
\Delta_{12} = \Delta_{23} = \Delta_{34} = \Delta_{14} = \sqrt{2},\quad \Delta_{13} = \Delta_{24} = 2,
\]
which agrees with \cref{tau_properties}\eqref{prop:X_0}. The image of $\phi$ is the subset of $\Gr(2,4)$ where $\delta\neq 0$, which we see includes $\Gr_{\ge 0}(2,4)$, verifying \cref{embedding_properties} in this case. Restricting $\phi^{-1}$ to $\Gr_{\ge 0}(2,4)$ gives a homeomorphism onto the subset of $\mathbb{R}^4$ of points $(a,b,c,d)$ where the $6$ polynomials $\Delta_{ij}$ in~\eqref{Gr24_Plueckers} are nonnegative. By \cref{thm:Gr}, these spaces are both homeomorphic to $4$-dimensional closed balls. The closures of cells in the cell decomposition of $\Gr_{\ge 0}(2,4)$ are obtained in $\RR^4$ by taking an intersection with the zero locus of some subset of the $6$ polynomials. The $0$-dimensional cells (corresponding to points of $\Gr_{\ge 0}(2,4)$ with only one nonzero Pl\"{u}cker coordinate) are
\[
(a,b,c,d)=(-2,1,-1,0),\hspace*{2pt} (0,-1,1,-2),\hspace*{2pt} (0,-1,1,2),\hspace*{2pt} (2,1,-1,0),\hspace*{2pt} (0,-1,-1,0),\hspace*{2pt} (0,1,1,0).
\]
In general, using the embedding $\phi$ we can describe $\Gr_{\ge 0}(k,n)$ as the subset of $\mathbb{R}^{k(n-k)}$ where some $\binom{n}{k}$ polynomials of degree at most $k$ are nonnegative.

\section{The totally nonnegative part of the unipotent radical of \texorpdfstring{$\GL_n(\RR)$}{GL(n)}}\label{sec:U}

\noindent Recall from \cref{intro:U} that $U$ denotes the unipotent group of upper-triangular matrices in $\GL_n(\RR)$ with $1$'s on the diagonal, and $U_{\ge 0}$ is its totally nonnegative part, where all minors are nonnegative. We also let $V\subset U$ be the set of $x\in U$ whose superdiagonal entries $x_{i,i+1}$ sum to $n-1$, and define $V_{\ge 0} := V\cap U_{\ge 0}$. We may identify $V_{\ge 0}$ with the link of the identity matrix $1$ in $U_{\ge 0}$. In this section, we prove the following result. It is a special case of a result of Hersh~\cite{Hersh}, who established the corresponding result in general Lie type, and in addition for all the lower-dimensional cells in the Bruhat stratification.
\begin{thm}[\cite{Hersh}]\label{thm:U}
The space $V_{\geq 0}$ is homeomorphic to an $\left(\binom{n}{2}-1\right)$-dimensional closed ball. The space $U_{\geq 0}$ is homeomorphic to a closed half-space in $\RR^{\binom{n}{2}}$.
\end{thm}

Let $e\in \gl_n(\RR)$ be the upper-triangular principal nilpotent element, which has $1$'s on the superdiagonal and $0$'s elsewhere. We wish to consider the flow on $V_{\geq 0}$ generated by $\exp(te)$, which we remark was used by Lusztig to show that $U_{\geq 0}$ is contractible~\cite[Section~4]{LusIntro}. However, we must take care to define a flow which preserves $V$ and not merely $U$. To this end, for $t > 0$, let $\rho(t)\in\GL_n(\RR)$ be the diagonal matrix with diagonal entries $(t^{n-1},t^{n-2},\dots,1)$. Note that $\rho$ is multiplicative, i.e., $\rho(s)\rho(t) = \rho(st)$. Define $a(t): U\to U$ by
\begin{align}\label{eq:a(t)}
a(t)\cdot x :=  \rho(1/t)\exp((t-1)e)x \rho(t).
\end{align}

\begin{lem}\label{multiplicative_group}
The map $a(\cdot)$ defines an action of the multiplicative group $\mathbb{R}_{>0}$ on $V$, i.e.,
\[
a(t)\cdot x\in V,\;\; a(1)\cdot x = x,\; \text{ and } \; a(s)\cdot (a(t)\cdot x) = a(st)\cdot x\quad \text{ for all $s,t > 0$ and $x\in V$}.
\]
\end{lem}

\begin{proof}
This can be verified directly, using the fact that for $s,t > 0$, $x\in U$, and $1 \le i,j \le n$,
\begin{equation}\label{eq:a(t)_entries}
(\rho(1/t)\exp((s-1)e)x \rho(t))_{i,j} = t^{i-n}t^{n-j}(\exp((s-1)e)x)_{i,j} = \frac{1}{t^{j-i}}\sum_{l=i}^j \frac{(s-1)^{l-i}}{(l-i)!}x_{l,j}.\hspace*{-8pt}\qedhere
\end{equation}
\end{proof}

We now introduce coordinates on $V$ centered at $\exp(e)$. Namely, for $x\in V$, define
\begin{align}\label{exp(e)_coordinates}
b_{i,j}(x) := c^{i-j}((j-i)!x_{i,j} - 1) \quad \text{ for }1 \le i < j \le n,
\end{align}
where $c>1$ is a fixed real number. Note that $b_{1,2}+b_{2,3}+\dots+b_{n-1,n}=0$. We use the $L^\infty$-norm $\|x\|_\infty := \max_{1 \le i < j \le n}|b_{i,j}(x)|$ in the $b_{i,j}$-coordinates. We have $\|\!\exp(e)\|_\infty = 0$. We also define the {\itshape totally positive part} $U_{>0}$ as the set of $x\in U_{\ge 0}$ such that every minor of the form $\det(x_{\{i_1 < \dots < i_k\}, \{j_1 < \dots < j_k\}})$ with $j_1 \ge i_1, \dots, j_k\ge i_k$ is nonzero, and let $V_{>0} := V\cap U_{>0}$.
\begin{lem}\label{lem:at}
Let $x\in V$.
\begin{enumerate}[(i)] 
\item\label{at1}
If $x\in V_{\geq 0}$ and $t > 1$, then $a(t) \cdot x \in V_{>0}$.
\item\label{at2}
If $x\neq\exp(e)$, then $t\mapsto \|a(t)\cdot x\|_\infty$ is a strictly decreasing function on $(0,\infty)$.
\end{enumerate}
\end{lem}

\begin{proof}
To prove part~\eqref{at1}, we must show that $\exp((t-1)e)x\in U_{>0}$ for $x\in V_{\ge 0}$ and $t > 1$. This follows by writing the relevant minors of $\exp((t-1)e)x$ via the Cauchy--Binet identity, using the fact that $\exp((t-1)e)\in U_{>0}$ (see~\cite[Proposition~5.9]{Lus2}). For part~\eqref{at2}, because $a(\cdot)$ is multiplicative (\cref{multiplicative_group}) it suffices to prove that $t\mapsto\|a(t)\cdot x\|_\infty$ is decreasing at $t=1$. Using the description of the entries of $a(t)$ given by setting $s=t$ in~\eqref{eq:a(t)_entries}, we get
\[
b_{i,j}(a(t)\cdot x) = b_{i,j}(x) + (t-1)(j-i)\left(\frac{b_{i+1,j}(x)}{c} - b_{i,j}(x)\right) + O((t-1)^2)
\]
as $t\to 1$, where we set $b_{i+1,j}(x) := 0$ if $i+1 = j$. Then for any $i < j$ with $|b_{i,j}(x)| = \|x\|_\infty$, we have $|b_{i+1,j}(x)/c| < |b_{i,j}(x)|$, and so $|b_{i,j}(a(t)\cdot x)|$ is decreasing at $t=1$.
\end{proof}

\begin{proof}[Proof of \cref{thm:U}]
Set $N:=\binom{n}{2}-1$, and consider a real $N$-dimensional vector space
\[W:=\{(w_{i,j})_{1\leq i<j\leq n} : w_{1,2}+\dots+w_{n-1,n}=0\}\]
equipped with the $L^\infty$-norm.
We have a diffeomorphism $b:V\to W$ defined by $b(x):=(b_{i,j}(x))_{1\leq i<j\leq n}$. Define the continuous map $\flow:\RR\times W\to W$ by setting $\act{t}{b(x)}:=b(a(\exp(t))\cdot x)$. Then \cref{multiplicative_group} and \cref{lem:at}\eqref{at2} imply that $\flow$ is a contractive flow.

Let us now show that the hypotheses of \cref{lem:top} hold, with $\RR^N = W$ and $\openball = b(V_{>0})$. \cref{lem:at}\eqref{at1} implies that $\closedball = b(V_{\ge 0})$, and also verifies~\eqref{eq:invariant}. Finally, $\openball$ is bounded because $V_{\ge0}$ is bounded, e.g., one can prove by induction on $j-i$ that for any $x\in V_{\ge 0}$,
\[
0 \le x_{i,j} \le (n-1)^{j-i} \quad \text{ for } 1 \le i < j \le n.
\]
(Alternatively, see~\cite[proof of Proposition~4.2]{Lus2}.) Thus \cref{lem:top} implies that $\closedball$ (and hence $V_{\ge 0}$) is homeomorphic to an $N$-dimensional closed ball.

For $U_{\ge 0}$, we use the dilation action of $\RR_{>0}$ on $U_{\geq 0}$, where $t \in \RR_{>0}$ acts by multiplying all entries $x_{i,j}$ on the $(j-i)$-th diagonal (above the main diagonal) by $t^{j-i}$. Therefore $U_{\geq 0}$ is homeomorphic to the open cone over the compact space $V_{\geq 0}$. That is, $U_{\geq0}$ is homeomorphic to the quotient space of $\RR_{\geq 0} \times V_{\geq 0}$ by the subspace $0 \times V_{\geq 0}$, with the identity matrix $1 \in U_{\geq 0}$ corresponding to the cone point.
\end{proof}

\begin{eg}
Let $n=3$. The trajectory in $U$ beginning at the point $x = \scalebox{0.75}{$\begin{bmatrix}1 & p & q \\ 0 & 1 & r \\ 0 & 0 & 1\end{bmatrix}$}\in U$ is
\[
a(t)\cdot x = \begin{bmatrix}
1 & (t+p-1)/t & (t^2 + (2r-2)t + 2q - 2r + 1)/2t^2 \\
0 & 1 & (t+r-1)/t \\
0 & 0 & 1
\end{bmatrix},
\]
which converges to $\exp(e) = \scalebox{0.75}{$\begin{bmatrix}1 & 1 & 1/2 \\ 0 & 1 & 1 \\ 0 & 0 & 1\end{bmatrix}$}$ as $t\to\infty$. The coordinates $b_{i,j}$ from~\eqref{exp(e)_coordinates} are
\[
b_{1,2}(a(t)\cdot x) = \frac{p-1}{ct},\quad b_{2,3}(a(t)\cdot x) = \frac{r-1}{ct},\quad b_{1,3}(a(t)\cdot x) = \frac{(2r-2)t + 2q-2r+1}{c^2t^2}.
\]
We can then try to verify \cref{lem:at} directly in this case (this is a nontrivial exercise).
\end{eg}

\section{The \cs amplituhedron}\label{sec:amplituhedron}

\noindent Let $k,m,n$ be nonnegative integers with $k+m\le n$ and $m$ even, and let $\shift,\tau\in\gl_n(\RR)$ be the operators from \cref{sec:glob-coord}. Let $\l_1\geq\dots\geq\l_n\in\RR$ be the eigenvalues of $\tau$ corresponding to orthogonal eigenvectors $\eigtau_1,\dots,\eigtau_n$. In this section, we assume that these eigenvectors have norm $1$. Recall from \cref{tau_properties}\eqref{tau_eigenvalues} that $\l_k>\l_{k+1}$. Since $m$ is even, we have $(-1)^{k+m-1} = (-1)^{k-1}$ and $\l_{k+m}>\l_{k+m+1}$.

Let $\Zsym$ denote the $(k+m)\times n$ matrix whose rows are $\eigtau_1, \dots, \eigtau_{k+m}$. By \cref{tau_properties}\eqref{prop:X_0}, the $(k+m)\times (k+m)$ minors of $\Zsym$ are all positive (perhaps after replacing $\eigtau_1$ with $-\eigtau_1$). We may also think of $\Zsym$ as a linear map $\RR^n\to\RR^{k+m}$. Since the vectors $u_1, \dots, u_n$ are orthonormal, this map takes $\eigtau_i$ to the $i$th unit vector $e_i\in\RR^{k+m}$ if $i\le k+m$, and to $0$ if $i > k+m$. Recall from \cref{intro:amplituhedron} that $\Zsym$ induces a map $\ZsymGr : \Grtnn\to\Gr(k,k+m)$, whose image is the {\itshape cyclically symmetric amplituhedron} $\mathcal{A}_{n,k,m}(\Zsym)$. We remark that if $g\in\GL_{k+m}(\RR)$, then $\mathcal{A}_{n,k,m}(g\Zsym)$ and $\mathcal{A}_{n,k,m}(\Zsym)$ are related by the automorphism $g$ of $\Gr(k,k+m)$, so the topology of $\mathcal{A}_{n,k,m}(\Zsym)$ depends only on the row span of $\Zsym$ in $\Gr(k+m,n)$.

\begin{proof}[Proof of \cref{thm:ampli}]
We consider the map $\phi:\Mat(k,n-k)\to\Gr(k,n)$ defined in~\eqref{not:embedding}. We write each $k\times (n-k)$ matrix $A\in\Mat(k,n-k)$ as $[A'\mid A'']$, where $A'$ and $A''$ are the $k\times m$ and $k\times (n-k-m)$ submatrices of $A$ with column sets $\{1,\dots,m\}$ and $\{m+1,\dots,n-k\}$, respectively. We introduce a projection map
\[
\proj:\Mat(k,n-k)\to\Mat(k,m),\quad A=[A'\mid A'']\mapsto A'.
\]
We claim that there exists an embedding $\emb:\amplZsym\hookrightarrow \Mat(k,m)$ making the following diagram commute:
\begin{equation}\label{commutative}
\begin{tikzcd}
  \Mat(k,n-k) \arrow[r,"\proj"] & \Mat(k,m)\\
  \Grtnn \arrow[u, hook,"\phi^{-1}"] \arrow[r,"\ZsymGr"]&   \amplZsym \arrow[u,dashrightarrow,hook,"\emb"]
\end{tikzcd}\;.
\end{equation}
Let $A=[A'\mid A'']\in\Mat(k,n-k)$ be a matrix such that $\phi(A)\in\Grtnn$. Then the element $\ZsymGr(\phi(A))$ of $\Gr(k,k+m)$ is the row span of the $k\times(k+m)$ matrix $[\Id_k\mid A']$, where $\Id_k$ denotes the $k\times k$ identity matrix. Thus $\mathcal{A}_{n,k,m}(\Zsym) = \ZsymGr(\Grtnn)$ lies inside the Schubert cell
\[
\{Y\in\Gr(k,k+m) : \Delta_{[k]}(Y)\neq 0\}.
\]
Every element $Y$ of this Schubert cell is the row span of $[\Id_k\mid A']$ for a unique $A'$, and we define $\emb(Y):=A'$. Thus $\emb$ embeds $\amplZsym$ inside $\Mat(k,m)$, and~\eqref{commutative} commutes.

Now we define
\[
\openball := \proj(\phi^{-1}(\Grtp)) \subset \Mat(k,m).
\]
We know from \cref{sec:proof} that $\phi^{-1}(\Grtp)$ is an open subset of $\Mat(k,n)$ whose closure $\phi^{-1}(\Grtnn)$ is compact. Note that $\proj$ is an open map (since it is essentially a projection $\RR^{k(n-k)}\to\RR^{km}$), so $\openball$ is an open subset of $\Mat(k,m)$. The closure $\closedball = \proj(\phi^{-1}(\Grtnn))$ of $\openball$ is compact. By~\eqref{commutative}, $\closedball$ is homeomorphic to $\mathcal{A}_{n,k,m}(\Zsym)$. 

Let $\flow:\RR\times\Mat(k,n-k)\to\Mat(k,n-k)$ be the map defined by~\eqref{eq:expts_dfn}, and define a similar map $f_0:\RR\times\Mat(k,m)\to\Mat(k,m)$ by
\[
f_0(t,A')_{i,j} := e^{t(\l_{k+j}-\l_i)}A'_{i,j}\quad\text{ for $1\leq i\leq k$ and $1\leq j\leq m$}.
\]
That is, $f_0(t,\proj(A)) = \pi(\act{t}{A})$ for all $t\in\RR$ and $A\in\Mat(k,n-k)$. We showed in \cref{sec:proof} that $\flow$ is a contractive flow, so $f_0$ is also a contractive flow. We also showed that
\[
\act{t}{\phi^{-1}(\Grtnn)} \subset \phi^{-1}(\Grtp) \quad \text{ for }t>0,
\]
and applying $\pi$ to both sides shows that
\[
f_0(t,\closedball) \subset \openball \quad \text{ for }t>0.
\]
Thus \cref{lem:top} applies to $\openball$ and $f_0$, showing that $\closedball$ (and hence $\mathcal{A}_{n,k,m}(\Zsym)$) is homeomorphic to a $km$-dimensional closed ball.
\end{proof}

\begin{eg}\label{eg_ampl}
Let $k=1$, $n=4$, $m=2$. We have
\begin{align*}
\tau = \begin{bmatrix}
0 & 1 & 0 & 1 \\
1 & 0 & 1 & 0 \\
0 & 1 & 0 & 1 \\
1 & 0 & 1 & 0
\end{bmatrix},\qquad
\begin{aligned}
  & u_1 = \left(\textstyle\frac12,\textstyle\frac12,\textstyle\frac12,\textstyle\frac12\right), && \lambda_1 = 2, \\
& u_2 = \left(\textstyle\frac1{\sqrt2}, 0, -\textstyle\frac1{\sqrt2}, 0\right), && \lambda_2 = 0, \\
& u_3 = \left(0,\textstyle\frac1{\sqrt2}, 0,-\textstyle\frac1{\sqrt2}\right), && \lambda_3 = 0, \\
& u_4 = \left(\textstyle\frac12, -\textstyle\frac12, \textstyle\frac12, -\textstyle\frac12\right), && \lambda_4 = -2,
\end{aligned}\qquad \Zsym= \begin{bmatrix}
\frac12 & \frac12 & \frac12 & \frac12 \\[6pt]
\frac1{\sqrt2} & 0 & -\frac1{\sqrt2} & 0 \\[6pt]
0 & \frac1{\sqrt2} & 0 & -\frac1{\sqrt2}
\end{bmatrix}.
\end{align*}
Note that this $\tau$ differs in the top-right and bottom-left entries from the one in \cref{eg_Gr24}, because $k$ is odd rather than even. Also, here the eigenvectors are required to have norm $1$. The embedding $\phi:\Mat(1,3)\hookrightarrow\Gr(1,4)$ sends a matrix $A:=\begin{bmatrix}a&b&c\end{bmatrix}$ to the line $\phi(A)$ in $\Gr(1,4)$ spanned by the vector
\[
v=u_1+au_2+bu_3+cu_4=\frac{1}{2}\left(1 + \sqrt{2}\hspace*{1.2pt}a + c, 1 + \sqrt{2}\hspace*{1.2pt}b - c, 1 - \sqrt{2}\hspace*{1.2pt}a + c, 1 - \sqrt{2}\hspace*{1.2pt}b - c\right).
\]
This line gets sent by $\ZsymGr$ to the row span of the matrix  $v\cdot \Zsym^T=\begin{bmatrix}1 & a & b\end{bmatrix}$. Finally, $\emb$ sends this element of $\Gr(1,3)$ to the matrix $\begin{bmatrix}a & b\end{bmatrix}$, so~\eqref{commutative} indeed commutes. 

In order for $\phi(A)$ to land in $\Gr_{\ge 0}(1,4)$, the coordinates of $v$ must all have the same sign, and since their sum is $2$, they must all be nonnegative:
\[
1+\sqrt2\hspace*{1.2pt} a+ c\geq 0,\quad 1+\sqrt2\hspace*{1.2pt} b-c\geq 0,\quad 1-\sqrt2\hspace*{1.2pt} a+c\geq0,\quad 1-\sqrt2\hspace*{1.2pt} b-c\geq0.
\]
These linear inequalities define a tetrahedron in $\R^3\simeq\Mat(1,3)$ with the four vertices $\left(0,\pm\sqrt2,-1\right),\left(\pm\sqrt2,0,1\right)$. The projection $\proj=\emb\circ\ZsymGr\circ\phi$ sends this tetrahedron to a square in $\R^2\simeq\Mat(1,2)$ with vertices $\left(0,\pm\sqrt2\right), \left(\pm\sqrt2,0\right)$. This square is a $km$-dimensional ball, as implied by \cref{thm:ampli}. We note that when $k=1$, the amplituhedron $\ampl$ (for any $(k+m)\times n$ matrix $Z$ with positive maximal minors) is a cyclic polytope in the projective space $\Gr(1,m+1) = \mathbb{P}^m$~\cite{sturmfels_88}, and is therefore homeomorphic to a $km$-dimensional closed ball. The case of $k\geq 2$ and $Z\neq\Zsym$ remains open.
\end{eg}

\section{The compactification of the space of electrical networks}\label{sec:E}

\subsection{A slice of the totally nonnegative Grassmannian}

We recall some background on electrical networks, and refer the reader to~\cite{Lam} and \cref{eg_NC} for details. Let $\Lin$ have basis vectors $e_I$ for $I\in \binom{[2n]}{n-1}$, and let $\Proj$ denote the corresponding projective space. We define
\[
\odd := \{2i-1 : i\in[n]\},\quad \even := \{2i : i\in [n]\}.
\]
Let $\NC$ denote the collection of \emph{non-crossing partitions} of $\odd$, i.e., set partitions of $\odd$ such that there do not exist $i < j < i' < j'$ in $\odd$ and distinct parts $I$ and $J$ with $i,i'\in I$ and $j,j'\in J$. Each $\sigma\in\NC$ comes with a \emph{dual non-crossing partition} (or \emph{Kreweras complement}) $\sigmadual$ of $\even$, defined to be the coarsest non-crossing partition of $\even$ such that $\sigma \cup \sigmadual$ is a non-crossing partition of $[2n]$. We call a subset $I\in\binom{[2n]}{n-1}$ \emph{concordant} with $\sigma$ if every part of $\sigma$ and every part of $\sigmadual$ contains exactly one element not in $I$. Let $A_\sigma\in\Lin$ be the sum of $e_I$ over all $I$ concordant with $\sigma$, and let $\Hspace$ be the linear subspace of $\Proj$ spanned by the images of $A_\sigma$ for $\sigma\in\NC$.

Identifying $\Grrr$ with its image under the Pl\"ucker embedding, we consider the subvariety $\X_n:=\Grrr\cap\Hspace$. In~\cite[Theorem~5.8]{Lam}, an embedding
\begin{align}\label{not:iota}
\iota: E_n \simeq \X_n\cap\Grtn \hookrightarrow \Gr_{\ge 0}(n-1,2n)
\end{align}
was constructed, identifying the \emph{compactification of the space of planar electrical networks with $n$ boundary vertices} $E_n$ with the compact space $\X_n\cap\Grtn$. We will need the following property of $(E_n)_{>0} := \X_n\cap\Gr_{>0}(n-1,2n)$.

\begin{prop}\label{prop:elec}
The space $(E_n)_{>0}$ is diffeomorphic to $\RR_{>0}^{\binom{n}{2}}$, and the inclusion $(E_n)_{>0}\hookrightarrow \Gr_{>0}(n-1,2n)$ is a smooth embedding.
\end{prop}
Here $\Gr_{>0}(n-1,2n) \subset \Gr(n-1,2n)$ is an open submanifold diffeomorphic to $\RR_{>0}^{(n-1)(n+1)}$.
\begin{proof}
We recall from~\cite[Theorem~4]{CIM} that each point in $(E_n)_{>0}=\Omega_n^+$ is uniquely represented by assigning a positive real number (the \emph{conductance}) to each edge of a well-connected electrical network $\Gamma$ with $\binom{n}{2}$ edges. This gives a parametrization $(E_n)_{>0} \simeq \RR_{>0}^{\binom{n}{2}}$.  The construction $\Gamma \mapsto N(\Gamma)$ of~\cite[Section 5]{Lam} sends $\Gamma$ to a weighted bipartite graph $N(\Gamma)$ embedded into a disk compatibly with the inclusion~\eqref{not:iota}.  The edge weights of $N(\Gamma)$ are monomials in the edge weights of $\Gamma$.  Furthermore, the underlying bipartite graph $G$ of $N(\Gamma)$ parametrizes $\Gr_{>0}(n-1,2n)$. That is, we can choose a set of $(n-1)(n+1)$ edges of $G$, so that assigning arbitrary positive edge weights to these edges and weight 1 to the remaining edges induces a parametrization $\Gr_{>0}(n-1,2n) \simeq \RR_{>0}^{(n-1)(n+1)}$ (see~\cite{Pos} or~\cite{TalaskaLe}). It follows that the inclusion $\RR_{>0}^{\binom{n}{2}} \simeq (E_n)_{>0} \hookrightarrow \Gr_{>0}(n-1,2n) \simeq \RR_{>0}^{(n-1)(n+1)}$ is a monomial map, and in particular a homomorphism of Lie groups. The result follows.
\end{proof}

\subsection{Operators acting on non-crossing partitions}

For each $i\in[2n]$, we define $\Up$ and $\Down$ in $\gl_{\binom{2n}{n-1}}(\RR)$ by
\[
\Up(\e_I):=
  \begin{cases}
    \e_{I\cup\{i+1\}\setminus\{i\}}, &\text{if $i\in I$, $i+1\notin I$;}\\
    0,&\text{otherwise;}  
  \end{cases}\quad
  \Down(\e_I):=
  \begin{cases}
    \e_{I\cup\{i-1\}\setminus\{i\}}, &\text{if $i\in I$, $i-1\notin I$;}\\
    0,&\text{otherwise.}  
  \end{cases}
\]
Here the indices are taken modulo $2n$.

For $i\in\odd$, we let $\bigpart(i)\in\NC$ be the non-crossing partition which has two parts, namely $\{i\}$ and $\odd\setminus \{i\}$. For $i\in\even$, we let $\smallpart(i)\in\NC$ be the non-crossing partition with $n-1$ parts, one of which is $\{i-1,i+1\}$ and the rest being singletons. Given $\sigma\in\NC$ and $i\in [2n]$, we define the noncrossing partition $\isolate\in\NC$ as the common refinement of $\sigma$ and $\bigpart(i)$ if $i$ is odd, and the common coarsening of $\sigma$ and $\smallpart(i)$ if $i$ is even. The following combinatorial lemma is essentially~\cite[Proposition~5.15]{Lam}, and can be verified directly.
\begin{lem}\label{lemma:UD}
  For all $i\in[2n]$, we have
  \[(\Up + \Down)(\A_\sigma)=
    \begin{cases}
      0, &\text{if $\sigma=\isolate$;}\\
      \A_{\isolate},&\text{otherwise.}
    \end{cases}\]
\end{lem}

\begin{eg}\label{eg_NC}
  Let $n:=3$ and $\sigma:=\{\{1,3\},\{5\}\}\in\NC$, so that $\sigmadual=\{\{2\},\{4,6\}\}$, $\isol1=\isol3=\{\{1\},\{3\},\{5\}\}$, $\isol2=\isol5=\sigma$, and $\isol4=\isol6=\{\{1,3,5\}\}$. Abbreviating $e_{\{a,b\}}$ by $e_{ab}$, we have
\[
\A_\sigma=e_{14}+e_{16}+e_{34}+e_{36}.
\]
Note that $\sigma\neq\isol1$ and
\[
(u_1 + d_1)(A_\sigma) = (e_{24} + e_{26}) + (e_{46}) = A_{\isol1},
\]
in agreement with \cref{lemma:UD} (since the dual of $\isol1$ is $\{\{2,4,6\}\}$). Similarly, we have $\sigma = \isol2$ and
\[
(u_2 + d_2)(A_\sigma) = 0 + 0 = 0.\qedhere
\]
\end{eg}

We define the operator $\UD:=\sum_{i=1}^{2n}u_i + d_i \in \gl_{\binom{2n}{n-1}}(\RR)$.

\begin{lem}\label{exp(ts)_En}
Let $X\in E_n$. We have $\exp(t\tau)\cdot X \in (E_n)_{>0}$ for all $t > 0$.
\end{lem}

\begin{proof}
This follows from \cref{cor:expttau_preserves_Grtp}, once we show that $\exp(t\tau)\cdot X \in \Hspace$ for $X \in \X_n$ and $t \in \RR$. To do this, we identify $\Grrr$ with its image in $\Proj$ under the Pl\"ucker embedding sending $X \in \Grrr$ to $\sum_{I\in{\binom{[2n]}{n-1}}}\Plucker_I(X)e_I\in\Proj$.  Then for any $X\in\X_n$, we have a smooth curve $t\mapsto\exp(t\tau)\cdot X$ in $\Proj$. As in \eqref{expts_Pluckers}, we find that
\[
\exp(t\tau)\cdot X = X+t\UD(X)+O(t^2)\quad \text{in }\Proj
\]
as $t\to 0$.  Therefore $\exp(t\tau)\cdot X$ is an integral curve for the smooth vector field on $\Proj$ defined by the infinitesimal action of $\UD$. By \cref{lemma:UD}, this vector field is tangent to $\Hspace$, so $\exp(t\tau)\cdot X\in\Hspace$ for all $t\in\RR$.
\end{proof}

\begin{proof}[Proof of \cref{thm:En}]
We are identifying $(E_n)_{>0}$ as a subset $\Gr_{>0}(n-1,2n)$ via the smooth embedding of \cref{prop:elec}. In turn, $\Gr_{>0}(n-1,2n)$ is smoothly embedded inside $\Mat(n-1,n+1)$ by the map $\phi^{-1}$ defined in~\eqref{not:embedding}. Thus $\openball:=\phi^{-1}((E_n)_{>0})\subset \Mat(n-1,n+1)$ is a smoothly embedded submanifold of $\Mat(n-1,n+1)$ of dimension $\binom{n}{2}$. The map $\phi^{-1}$ sends the compact set $E_n$ homeomorphically onto its image $\phi^{-1}(E_n)$.  Since $(E_n)_{>0}$ is dense in $E_n$, we have that $\phi^{-1}(E_n)$ equals the closure $\closedball$ of $\openball$. Let $\flow:\RR\times\Mat(n-1,n+1)\to \Mat(n-1,n+1)$ be the map defined by~\eqref{eq:expts_dfn}. We showed in \cref{sec:proof} that $\flow$ is a contractive flow, and \cref{exp(ts)_En} implies that~\eqref{eq:invariant} holds for our choice of $\openball$ and $\flow$. Thus \cref{lem:top} applies, completing the proof.
\end{proof}

\bibliographystyle{alpha}
\bibliography{ball_bib}

\end{document}